\documentclass[
final,
nomarks
]{dmtcs-episciences}

\usepackage[utf8]{inputenc}
\usepackage{subfigure}

\newtheorem{thm}{Theorem}

\newtheorem{lem}[thm]{Lemma}
\newtheorem{cor}[thm]{Corollary}

\author{Yan Li
  \and Xin Zhang\thanks{This author is the corresponding
    author. Email: xzhang@xidian.edu.cn.\\
    \indent ~~~This author was supported by the National Natural Science Foundation of China (No.\,11871055).}
 }
\title[The structure and the list 3-dynamic coloring of outer-1-planar graphs]{The structure and the list 3-dynamic coloring of outer-1-planar graphs}
\affiliation{
  School of Mathematics and Statistics, Xidian University, China}
\received{2019-10-22}
\revised{2021-01-29}
\accepted{2021-08-06}

\begin{document}
\publicationdetails{23}{2021}{3}{4}{5860}
\maketitle
\begin{abstract}
  An outer-1-planar graph is a graph admitting
a drawing in the plane so that all vertices
appear in the outer region of the drawing and every edge crosses at most one other edge. This paper establishes the local structure of outer-1-planar graphs by proving that each outer-1-planar graph contains one of the seventeen fixed configurations, and the list of those configurations is minimal in the sense that for each fixed configuration there exist outer-1-planar graphs containing this configuration that do not contain any of another sixteen configurations. There are two interesting applications of this structural theorem. First of all, we conclude that every (resp.\,maximal) outer-1-planar graph of minimum degree at least 2 has an edge with the sum of the degrees of its two end-vertices being at most 9 (resp.\,7), and this upper bound is sharp. On the other hand, we show that the list 3-dynamic chromatic number of every outer-1-planar graph is at most 6, and this upper bound is best possible.
\end{abstract}

%
%

\section{Introduction}
\label{sec:in}
A graph is \emph{1-planar} if it can be drawn in the plane such that each edge is crossed at most once. The family of 1-planar graphs is among the most investigated graph families within the so-called ``beyond planar graphs'', see \cite{DLM19,KLM17}. In this paper, we focus on a subclass of 1-planar graphs, in particular, outer-1-planar graphs.
A graph is said to be \emph{outer 1-planar} if it has a drawing in the plane so that all vertices
appear in the outer region of the drawing and every edge crosses at most one other edge; such a drawing is called an \emph{outer-1-plane graph} and the outer region of the drawing is called the \emph{outer boundary} of $G$. An outer-1-planar graph is \emph{maximal} if adding any edge (not multi-edge) to it will disturb the outer-1-planarity.
The concept of outer-1-planar graphs was first introduced by
\cite{Eggleton} who called them \emph{outerplanar graphs with edge
  crossing number one}, and also investigated under the notion of
\emph{pseudo-outerplanar graphs}, see \cite{TZ,LTC,ZPOPG}. Note that
outer-1-planar graphs are planar, see \cite{ABBGHNR2,ZPOPG}. Various
topics on outer-1-planar graphs including the recognition, see
\cite{ABBGHNR,Hong.etal}, drawing, see \cite{DE,GLM15}, structure, see
\cite{ZLLZ,ZPOPG} and coloring, see \cite{Chen19,TZ,LTC,LZ20,edge,Z17,ZL19,Z20}, are explored.

If $\mathcal{G}$ is a class of graphs such that each graph $G\in \mathcal{G}$ contains an edge $uv$ with $d(u)+d(v)\leq C$, where $C$ is a constant independent of $G$, then we say that $\mathcal{G}$ contains \emph{light edge}, or $\mathcal{G}$ is an \emph{edge-light graph class}. An edge $uv$ is of \emph{type $(a,\le b)$} if $d(u)=a$ and $d(v)\le b$. Finding edge-light graph classes is an interesting topic in the literature.

\cite{Kotzig55} proved that each 3-connected planar graph contains an edge
$uv$ such that $d(u)+d(v)\leq 13$ and this bound is sharp. \cite{FM07} showed that each 3-connected 1-planar graph contains an edge $uv$ such that $\max\{d(u),d(v)\}\leq 20$ and the bound 20 is sharp. \cite{LS17} proved that each 3-connected 1-planar graph $G$ contains an edge $uv$ of type $(3,\leq 22)$, $(4,\le 13)$, $(5,\le 9)$, $(6,\le 8)$ or $(7,7)$ and thus $d(u)+d(v)\leq 25$. Replacing the 3-connectedness with a condition on the minimum degree, \cite{HS12} proved that each 1-planar graph of minimum degree at least 4 contains an edge
of type $(4,\le 13)$, $(5,\le 9)$, $(6,\le 8)$ or $(7,7)$, where the bound 9, 8 and 7 in the last three types are sharp. Very recently, \cite{BDHS20} showed that each  1-planar graph of minimum degree at least 3 contains an edge $uv$ such that $\max\{d(u),d(v)\}\leq 29$. \cite{NZ20}
showed that each 1-planar graph of minimum degree at least 3 contains an edge
of type $(3,\le 23)$, $(4,\le 11)$, $(5,\le 9)$, $(6,\le 8)$ or
$(7,7)$. It is not clear whether the bound 23 and 11 in the first two
types are sharp, and the authors conjectured that they may be improved to 20 and 10, respectively.

For subclasses of planar graphs, it is well known, see \cite{Wang1999}, that each outerplanar graph of minimum degree at least 2 contains en edge $uv$ such that $d(u)+d(v)\leq 6$ and this bound is sharp.
\cite{ZLLZ} proved that (i) each outer-1-planar graph of minimum degree at least 2 has an edge $uv$ such that $d(u)+d(v)\leq 9$
 and the bound is sharp; (ii) every maximal outer-1-planar graph $G$ has an edge $uv$ such that $d(u)+d(v)\leq 7$ and the bound is sharp.

The aim of this paper is to improve the above two results of \cite{ZLLZ} to a more detailed form, which not only confirms the existence of such a light edge but also shows in which configuration it is contained (see Theorem \ref{str-1} in Section \ref{sec-2}). Actually, our result implies that each outer-1-planar graph of minimum degree at least 2 contains an edge
of type $(2,\le 7)$ or $(3,3)$, and each maximal outer-1-planar graph  contains an edge
of type $(2,\le 5)$ or $(3,3)$, and all bounds are sharp.

On the other hand, this structural theorem is applicable to an
interesting problem the so-called list 3-dynamic coloring of graphs,
which has many applications to the channel assignment problems, see \cite{ZB18,ZB20}. For the continuity of this paper, we introduce the list 3-dynamic coloring in Section \ref{sec-3}, where we give a sharp upper bound for the list 3-dynamic chromatic number of outer-1-planar graphs  (see Theorem \ref{list-3-dynamic} in Section \ref{sec-3}).

\section{Structural Theorem}\label{sec-2}

If an outer-1-plane graph $G$ is 2-connected, then we denote by $v_1, v_2, \ldots, v_{|G|}$ the vertices in the outer boundary of $G$ consecutively in a clockwise order, and then
let $\mathcal{V}[v_i,v_j]=\{v_i, v_{i+1}, \ldots, v_j\}$ and $\mathcal{V}(v_i,v_j)=\mathcal{V}[v_i,v_j]\backslash \{v_i,v_j\}$ (we take modulo $|G|$ for the subscripts).
The subgraph of $G$ induced by $\mathcal{V}[v_i,v_j]$ is denoted by $G[v_i,v_j]$.

If there is no edge between $\mathcal{V}(v_i,v_l)$ and $\mathcal{V}(v_l,v_i)$, where $i<l$, then we denote by
 $\widehat{G}[v_i,v_l]$ the graph derived from $G[v_i,v_l]$ via adding an edge $v_iv_l$ if it does not originally exist in $G$. Clearly,  $\widehat{G}[v_i,v_l]$ is also a 2-connnected outer-1-plane graph.

Given a vertex set $\mathcal{V}[v_i,v_j]$ with $i\neq j$, if $j=i+1$ (mod $|G|$) and $v_iv_j\not\in E(G)$, then we call it a \emph{non-edge}, and if $v_k v_{k+1}\in E(G)$ for all $i\leq k<j$, then we call it a \emph{path}. If $xy\in E(G)$ and $x,y$ are not two consecutive vertices in the outer boundary of $G$, then we call $xy$ a \emph{chord}. A chord that is crossed in $G$ is called a \emph{crossed chord}. The set of  chords $xy$ with $x,y\in \mathcal{V}[v_i,v_j]$ is denoted by $\mathcal{C}[v_i,v_j]$

\begin{lem}\label{path}\cite[Claim 1]{ZPOPG}
Let $G$ be a $2$-connected outer-$1$-plane graph and let $v_i,v_j$ be vertices of $G$. If each chord in $\mathcal{C}[v_i,v_j]$ is not crossed and there is no edge between
$\mathcal{V}(v_i,v_j)$ and $\mathcal{V}(v_{j},v_{i})$, then $\mathcal{V}[v_i,v_j]$ is either a path or a non-edge.
\end{lem}

If $G$ contains a configuration $G_i$ as shown in Fig.\,\ref{structure-1} such that any hollow (resp.\,solid black) vertex has the degree in $G$ at least (resp.\,exactly) the number of edges incident with it there,
and any solid grey vertex has the degree in $G$ as marked by Fig.\,\ref{structure-1},
then we say that $G$ \emph{contains} $G_i$.
For two vertices $v_a,v_b\in V(G)$, saying $G[v_a,v_b]$ \emph{properly contains} $G_i$, we mean that $G[v_a,v_b]$ contains $G_i$ so that neither $v_a$ nor $v_b$ corresponds to a solid black or grey vertex in the picture of $G_i$.

\begin{figure}
  \centering
  \includegraphics[width=15cm]{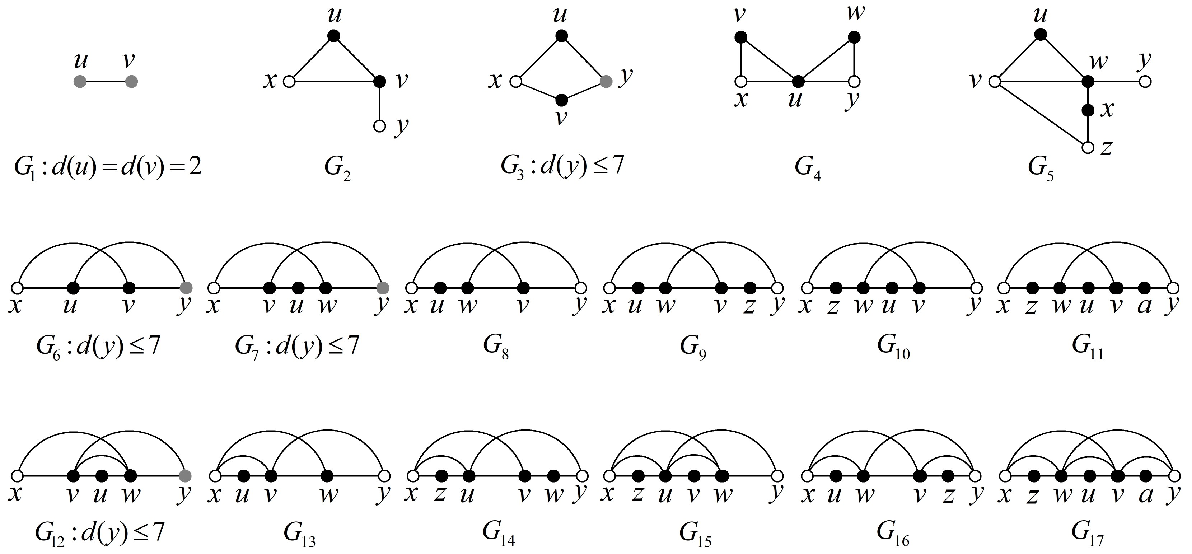}
  \caption{Local structures in an outer-1-planar graph $G$ with $\delta(G)\geq 2$}\label{structure-1}
\end{figure}

\begin{lem}\label{n-path}
If $G$ contains a path $v_{1} v_{2} \ldots v_{n}$ with $n\geq 5$ such that $v_iv_{i+1}$ is not a chord for each $1\leq i\leq n-1$, and
there are no chords $v_{i}v_{j}$ and $v_{k}v_{l}$ such that
$1\leq i<k<j<l\leq n$, then $G[v_{1},v_{n}]$ properly contains $G_{1}, G_{2}$ or  $G_{4}$.
\end{lem}

\begin{proof}
We prove it by induction on $n$.

\textbf{Case 1.} $n=5$.

If $\mathcal{C}[v_{1},v_{n}]=\emptyset$, then $d(v_{2})=d(v_{3})=2$ and $G_{1}$ is properly contained in $G[v_{1},v_{n}]$. Otherwise,
choose a chord $v_{i}v_{j}\in \mathcal{C}[v_{1},v_{n}]$ so that $\mathcal{C}[v_{i},v_{j}]={\{v_{i}v_{j}}\}$ and $i<j$. Now, if $j-i\geq 3$, then $d(v_{i+1})=d(v_{i+2})=2$ and $G[v_{1},v_{n}]$ properly contains $G_{1}$. If $j-i=2$, then $d(v_{i+1})=2$, and either $i\neq 1$ or $j\neq n$. By symmetry, we assume the latter. If $d(v_{j})=3$, then $G[v_{1},v_{n}]$ properly contains $G_{2}$. If $d(v_{j})=4$, then by symmetry, we consider two subcases. First, if $j=3$, then $v_{3}v_{5}\in E(G)$ and $d(v_{4})=2$, which implies the proper containment of $G_{4}$. Second, if $j=4$, then $i=2$ and $v_{1}v_{4}\in E(G)$. This concludes that $d(v_{2})=3$ and $G_{2}$  is properly contained in $G[v_{1},v_{n}]$.

\textbf{Case 2.} $n\geq 6$.

Suppose that we have proved the lemma for every $n'$ with $5\leq n'<n$.

We assume that there is a chord $v_{i}v_{j}\in \mathcal{C}[v_{1},v_{n}]$ so that $j-i=2$, otherwise we can finish the proof as in Case 1. Assume, without loss of generality, that $j\neq n$. If $d(v_{j})=3$, then $G[v_{1},v_{n}]$ properly contains $G_{2}$. Hence we assume $d(v_{j})\geq 4$. Therefore, there exists a chord $v_{j}v_{k}$ with $1\leq k<i$ or $j<k\leq n$.

If $1\leq k<i$, then $i\neq 1$, which implies $d(v_{i})\geq 4$, because otherwise $d(v_{i})=3$ and thus $G_{2}$ is properly contained. Whereafter, there is a chord $v_{i}v_{t}$ with $k\leq t<i$. Hence, there exists either a chord $v_{j}v_{k}$ with $j<k\leq n$, or a chord $v_{i}v_{t}$ with $1\leq t<i$. By symmetry, we assume the former.

If $k-j\geq 4$, then $G[v_{j},v_{k}]$ properly contains $G_{1},G_{2} $ or $ G_{4}$ by the induction hypothesis, and so does $G[v_{1},v_{n}]$, since any vertex in $\mathcal{V}(v_{j},v_{k})$ has the same degree both in $G[v_{j},v_{k}]$ and $G[v_{1},v_{n}]$.

If $k-j=3 $, then either $d(v_{j+1})=d(v_{j+2})=2$, or $d(v_{j+1})=2$, $d(v_{j+2})=3$ and $v_jv_{j+2}\in E(G)$, or $d(v_{j+1})=3$, $d(v_{j+2})=2$ and $v_{j+1}v_{k}\in E(G)$. In each case, we conclude that $G[v_{j},v_{k}]$ properly contains $G_{1}$ or $G_{2}$, and so does $G[v_{1},v_{n}]$.

If $k-j=2$, then $d(v_{j+1})=2$. At this moment, if $d(v_{j})=4$, then $G_{4}$  is properly contained in $G[v_{1},v_{n}]$. Otherwise, there is a chord $v_{j}v_{s}$ such that $k<s\leq n$ or $1\leq s<i$. By symmetry, assume the former. If $s-k\geq 2$, then $|\mathcal{V}[v_{j},v_{s}]|\geq 5$ and by the induction hypothesis, $G[v_{j},v_{s}]$ properly contains $G_{1},G_{2}$ or $ G_{4}$, and so does $G[v_{1},v_{n}]$.
If $s-k=1$, then $d(v_{k})=3$ and $G_{2}$ is properly contained in $G[v_{1},v_{n}]$.
\end{proof}

Let $G$ be a 2-connected outer-1-plane graph. By $\mathfrak{D}_1$ and $\mathfrak{D}_2$, we define two possible properties for the drawing of $G$. They are stated as follows.

\begin{description}
  \item[$\mathfrak{D}_1$] If $G$ contains $G_3$, then the graph derived from $G$ by adding a new edge between $u$ and $v$ in that picture is still outer-1-planar;
  \item[$\mathfrak{D}_2$] If $G$ contains $G_i$ for some $6\leq i\leq 17$, then the picture of $G_i$ in Fig.\,\ref{structure-1} corresponds to a partial drawing (up to inversion) of $G$ in the plane.
\end{description}

\begin{thm}\label{2-connected}
Let $G$ be a $2$-connected outer-$1$-plane graph and let $v_1,v_2,\ldots,v_n$ ($n=|G|$) be its vertices appearing in the outer boundary of $G$ consecutively in that order.
\begin{description}
  \item[(1)] If $n=4$, then $G[v_1,v_4]$ properly contains $G_1$ or $G_2$, unless $v_1v_3$ crosses $v_2v_4$ and $v_1v_2,v_3v_4\in E(G)$;
  \item[(2)] If $n=5$, then $G[v_1,v_5]$ properly contains one of the configurations among $G_{1}-G_{4},G_6,G_8,G_{13}$ such that $\mathfrak{D}_1$ and $\mathfrak{D}_2$ hold unless $\mathcal{V}[v_1,v_5]$ is a path and $v_1v_4,v_2v_5\in E(G)$;
  \item[(3)] If $n\geq 6$, then $G[v_1,v_n]$ properly contains one of the configurations among $G_{1}-G_{17}$ such that $\mathfrak{D}_1$ and $\mathfrak{D}_2$ hold.
\end{description}
\end{thm}

\begin{proof}
If no crossing appears in $G$, then $v_{1} v_{2}\cdots v_{n}$ is a path by the 2-connectedness of $G$, and thus $G[v_1,v_n]$ properly contains $G_1$, $G_2$ or $G_4$ by Lemma \ref{n-path} if $n\geq 5$. On the other hand, if $n=4$, then $G[v_1,v_4]$ properly contains $G_2$ if $v_1v_3\in E(G)$ or $v_2v_4\in E(G)$, and $G_1$ otherwise.
Hence in the following we assume that there is a pair of crossed chords $v_{i}v_{j}$ and $v_{k}v_{l}$ with $1\leq i<k<j<l\leq n$,

\textbf{Case 1. $n=4$.}

Suppose that $v_1v_3$ crosses $v_2v_4$. If $v_1v_2\not\in E(G)$, then $v_2v_3\in E(G)$ by the 2-connectedness of $G$. Therefore, $G[v_1,v_4]$ properly contains $G_1$ if $v_3v_4\not\in E(G)$, and $G_2$ otherwise. Hence we assume $v_1v_2\in E(G)$. By symmetry, it is also assumed that $v_3v_4\in E(G)$. This is in accordance with the
excluded case listed in (1).

\textbf{Case 2.} $n=5$.

By symmetry, we analyse three subcases as follows.

\textbf{Subcase 2.1.} $v_{1}v_{3}$ crosses $v_{2}v_{4}$.

By (1), $\hat{G}[v_1,v_4]$ properly contains $G_1$ or $G_2$ (and so does $G[v_1,v_4]$), unless $v_1v_3$ crosses $v_2v_4$ and $v_1v_2,v_3v_4\in E(G)$, in which case we have $d(v_4)\leq 4$. Therefore, $G[v_1,v_5]$ properly contains $G_3$ if $v_2v_3\not\in E(G)$, and $G_6$ otherwise. Moreover, $\mathfrak{D}_1$ and $\mathfrak{D}_2$ hold.

\textbf{Subcase 2.2.}  $v_{1}v_{3}$ crosses $v_{2}v_{5}$.

By the 2-connectedness of $G$, $v_3v_4,v_4v_5\in E(G)$.

If $v_2v_3\not\in E(G)$, then $d(v_{3})=d(v_{4})=2$ if $v_{3}v_{5}\not\in E(G)$, and $d(v_{3})=3,d(v_{4})=2$ if $v_{3}v_{5}\in E(G)$. Therefore, $G[v_{1},v_{n}]$ properly contain $G_1$ in the former case, and $G_2$ in the latter case.

If $v_2v_3\in E(G)$ and $v_1v_2\not\in E(G)$, then $d(v_2)=d(v_4)=2$, $d(v_3)\leq 4$, and thus  $G_3$ is properly contained in $G[v_{1},v_{n}]$. We confirm that $\mathfrak{D}_1$ holds by
showing that $G+v_2v_4$ is still outer-1-planar. Actually, adjusting the order of the vertices in the outer boundary from $v_1,v_2,v_3,v_4,v_5$ to $v_1,v_3,v_2,v_4,v_5$, we obtain an outer-1-planar drawing of the graph $G+v_2v_4$.

If $v_1v_2,v_2v_3\in E(G)$, then $G[v_{1},v_{n}]$ properly contains $G_8$ if $v_3v_5\not\in E(G)$, and $G_{13}$ if $v_3v_5\in E(G)$.

\textbf{Subcase 2.3.}  $v_{1}v_{4}$ crosses $v_{2}v_{5}$.

By the 2-connectedness of $G$, $v_2v_3,v_3v_4\in E(G)$.

If $v_4v_5\not\in E(G)$ (the case that $v_1v_2\not\in E(G)$ is symmetric), then $d(v_{3})=d(v_{4})=2$ if $v_{2}v_{4}\not\in E(G)$, and $d(v_{3})=2,d(v_{4})=3$ if $v_{2}v_{4}\in E(G)$. Therefore, $G[v_{1},v_{n}]$ properly contain $G_1$ in the former case, and $G_2$ in the latter case.

If $v_1v_2,v_4v_5\in E(G)$, then we meet the excluded case mentioned in (2).

\textbf{Case 3.} $n=6$.


If $|\mathcal{V}[v_i,v_l]|=4$, then by (1), $\hat{G}[v_i,v_l]$ properly contains $G_1$ or $G_2$ (and so does $G[v_i,v_l]$), unless $v_iv_k,v_jv_l\in E(G)$, in which case we have $d(v_i),d(v_l)\leq 5$. Since either $i\neq 1$ or $l\neq 6$, $G[v_1,v_6]$ properly contains $G_3$ if $v_kv_j\not\in E(G)$, and $G_6$ otherwise. Moreover, $\mathfrak{D}_1$ and $\mathfrak{D}_2$ hold trivially.

If $|\mathcal{V}[v_i,v_l]|=5$, then by (2),  $\hat{G}[v_i,v_l]$ properly contains one of the configurations among $G_{1}-G_{4},G_6,G_8,G_{13}$ (and so does $G[v_i,v_l]$) such that $\mathfrak{D}_1$ and $\mathfrak{D}_2$ hold, unless
$\mathcal{V}[v_i,v_l]$ is a path and $k=i+1,j=l-1$, in which case we have $d(v_i),d(v_l)\leq 4$. Since either $i\neq 1$ or $l\neq 6$, $G[v_1,v_6]$ properly contains $G_7$ if $v_kv_j\not\in E(G)$, and $G_{12}$ otherwise. Moreover, $\mathfrak{D}_2$ holds.

If $|\mathcal{V}[v_i,v_l]|=6$, then $i=1$ and $l=6$. By symmetry, we analyse four subcases as follows.

\textbf{Subcase 3.1.} $v_{1}v_{3}$ crosses $v_{2}v_{6}$.

By (1), $\hat{G}[v_3,v_6]$ properly contains $G_1$ or $G_2$ (and so does $G[v_3,v_6]$), unless $v_3v_5$ crosses $v_4v_6$ and $v_3v_4,v_5v_6\in E(G)$, in which case we have $d(v_3)\leq 5$. Hence $G[v_1,v_6]$ properly contains $G_3$ if $v_4v_5\not\in E(G)$, and $G_6$ otherwise. Moreover, $\mathfrak{D}_1$ and $\mathfrak{D}_2$ hold trivially.

\textbf{Subcase 3.2.} $v_{1}v_{4}$ crosses $v_{2}v_{6}$.

By the 2-connectedness of $G$, $v_2v_3,v_3v_4,v_4v_5,v_5v_6\in E(G)$. If $v_1v_2\not\in E(G)$, then $d(v_2)=d(v_3)=2$ if $v_2v_4\not\in E(G)$, and $d(v_2)=3$, $d(v_3)=2$ if $v_2v_4\in E(G)$. This implies that $G[v_1,v_6]$ properly contains $G_1$ in the former case, and $G_2$ in the latter case. Hence we assume $v_1v_2\in E(G)$. This implies that
$G[v_1,v_6]$ properly contains $G_{10}$ if $v_2v_4,v_4v_6\not\in E(G)$, $G_5$ if $|\{v_2v_4,v_4v_6\}\cap E(G)|=1$, and $G_{15}$ if $v_2v_4,v_4v_6\in E(G)$. Moreover, $\mathfrak{D}_2$ holds trivially.

\textbf{Subcase 3.3.} $v_{1}v_{4}$ crosses $v_{3}v_{6}$.

By the 2-connectedness of $G$, $v_1v_2,v_2v_3,v_4v_5,v_5v_6\in E(G)$. If $v_3v_4\not\in E(G)$, then $d(v_2)=d(v_3)=2$ if $v_1v_3\not\in E(G)$, and $d(v_2)=2$, $d(v_3)=3$ if $v_1v_3\in E(G)$, which implies that $G[v_1,v_6]$ properly contains $G_1$ in the former case, and $G_2$ in the latter case. Hence we assume $v_3v_4\in E(G)$. This implies that
$G[v_1,v_6]$ properly contains $G_{9}$ if $v_1v_3,v_4v_6\not\in E(G)$, $G_{14}$ if $|\{v_1v_3,v_4v_6\}\cap E(G)|=1$, and $G_{16}$ if $v_1v_3,v_4v_6\in E(G)$. Moreover, $\mathfrak{D}_1$ and $\mathfrak{D}_2$ hold trivially.

\textbf{Subcase 3.4.} $v_{1}v_{5}$ crosses $v_{2}v_{6}$.

By (1), $\hat{G}[v_2,v_5]$ properly contains $G_1$ or $G_2$ (and so does $G[v_2,v_5]$), unless $v_2v_4$ crosses $v_3v_5$ and $v_2v_3,v_4v_5\in E(G)$, in which case we have $d(v_2)\leq 5$. Hence $G[v_1,v_6]$ properly contains $G_3$ if $v_3v_4\not\in E(G)$, and $G_6$ otherwise. Moreover, $\mathfrak{D}_1$ and $\mathfrak{D}_2$ hold trivially.



\textbf{Case 4.} $n\geq 7$.

 Suppose that we have proved the lemma for every $n'$ with $6\leq n'<n$. 

\vspace{3mm}\noindent \textbf{Claim A.} \emph{If $v_{i}v_{j}$ crosses $v_{k}v_{l}$ ($1\leq i<k<j<l\leq n$), then $G[v_{1},v_{n}]$ properly contains at least one of the configurations among $G_{1}-G_{17}$ such that $\mathfrak{D}_1$ and $\mathfrak{D}_2$ hold, unless $4\leq |\mathcal{V}[v_{i},v_{l}]|\leq 5$, $k=i+1,j=l-1$, and $v_{i}v_{k},v_{j}v_{l}\in E(G)$,
in which case we say that $v_iv_j$ \underline{co-crosses} $v_kv_l$ in $G$.}

\begin{proof}

If $\max{\{|\mathcal{V}[v_{i},v_{k}]|,|\mathcal{V}[v_{k},v_{j}]|,|\mathcal{V}[v_{j},v_{l}]|}\}\geq 6 $, then we consider, without loss of generality, the case that $|\mathcal{V}[v_{i},v_{k}]|\geq 6$. Applying the
induction hypothesis to $\widehat{G}[v_i,v_k]$ (note that there is no edge between $\mathcal{V}(v_i,v_k)$ and $\mathcal{V}(v_k,v_i)$), we conclude that it properly contains one of the configurations among $G_{1}-G_{17}$ such that $\mathfrak{D}_1$ and $\mathfrak{D}_2$ hold, and so does $G[v_{1},v_{n}]$, since any vertex in $\mathcal{V}(v_i,v_k)$  has same degree both in $\widehat{G}[v_i,v_k]$ and $G[v_{1},v_{n}]$. Hence, we are only care about the case $\max{\{|\mathcal{V}[v_{i},v_{k}]|,|\mathcal{V}[v_{k},v_{j}]|,|\mathcal{V}[v_{j},v_{l}]|}\}\leq 5$.

\textbf{Subcase 4.1.} $\max{\{|\mathcal{V}[v_{i},v_{k}]|,|\mathcal{V}[v_{k},v_{j}]|,|\mathcal{V}[v_{j},v_{l}]|}\}=5$.

We only consider the case $|\mathcal{V}[v_{i},v_{k}]|= 5$, and the cases that $|\mathcal{V}[v_{k},v_{j}]|= 5$ or $|\mathcal{V}[v_{j},v_{l}]|= 5$ can be considered similarly.
By (2), $\widehat{G}[v_i,v_k]$ properly contains one of the configurations among $G_{1}-G_{4},G_6,G_8,G_{13}$ such that $\mathfrak{D}_1$ and $\mathfrak{D}_2$ hold (and so does $G[v_i,v_k]$) unless $\mathcal{V}[v_i,v_k]$ is a path such that $v_iv_{k-1},v_{i+1}v_k\in E(G)$. If $d(v_k)\leq 7$, then $G[v_1,v_n]$ properly contains $G_7$ if $v_{i+1}v_{k-1}\not\in E(G)$, and $G_{12}$ otherwise. It is easy to see that $\mathfrak{D}_2$ holds now.
Therefore, we assume that $d(v_k)\geq 8$. This implies that $j=k+4$ (note that $|\mathcal{V}[v_k,v_j]|\leq 5$)
and that $v_k$ is adjacent to $v_i,v_{k+1},v_{k+2},v_{k+3},v_j,v_l$. By (2), $\widehat{G}[v_k,v_j]$ properly contains one of the configurations among $G_{1}-G_{4},G_6,G_8,G_{13}$ (and so does $G[v_k,v_j]$), since $|\mathcal{V}[v_k,v_j]|=5$ and $\widehat{G}[v_k,v_j]$ cannot contain the excluded structure mentioned in (2).

\textbf{Subcase 4.2.} $\max{\{|\mathcal{V}[v_{i},v_{k}]|,|\mathcal{V}[v_{k},v_{j}]|,|\mathcal{V}[v_{j},v_{l}]|}\}=4$.

We only consider the case $|\mathcal{V}[v_{i},v_{k}]|= 4$, and the cases that $|\mathcal{V}[v_{k},v_{j}]|= 4$ or $|\mathcal{V}[v_{j},v_{l}]|= 4$ can be considered similarly.
By (1), $\widehat{G}[v_i,v_k]$ properly contains $G_1$ or $G_2$, unless $v_iv_{k-1}$ crosses $v_{i+1}v_k$ and $v_iv_{i+1},v_{k-1}v_k\in E(G)$. Since $|\mathcal{V}[v_k,v_j]|\leq 4$, $d(v_k)\leq 7$. Therefore, $G[v_1,v_n]$ properly contains $G_3$ if $v_{i+1}v_{k-1}\not\in E(G)$, and $G_6$ otherwise. Moreover, $\mathfrak{D}_1$ and $\mathfrak{D}_2$ hold.

\textbf{Subcase 4.3.} $|\mathcal{V}[v_{i},v_{k}]|=3$, and $|\mathcal{V}[v_{k},v_{j}]|,|\mathcal{V}[v_{j},v_{l}]|\leq 3$.

By the 2-connectedness of $G$, $\mathcal{V}[v_i,v_k]$ is a path.

Assume first that $|\mathcal{V}[v_{k},v_{j}]|=3$. Clearly,  $\mathcal{V}[v_k,v_j]$ is also a path since $G$ is 2-connected.

If exactly one from $v_iv_k$ and $v_kv_j$, say $v_iv_k$, is an edge of $G$, then $d(v_{i+1})=d(v_{k+1})=2$ and $d(v_{k})=4$,  which implies the proper containment of $G_5$ in $G[v_{1},v_{n}]$.

If $v_{i}v_{k}\not\in E(G)$ and $v_{k}v_{j}\not\in E(G)$, then we look at $|\mathcal{V}[v_{j},v_{l}]|$.
If $|\mathcal{V}[v_{j},v_{l}]|=3$, then $\mathcal{V}[v_j,v_l]$ is  a path by the 2-connectedness of $G$, and thus $G[v_1,v_n]$ properly contains $G_{5}$ if $v_jv_l\in E(G)$, and $G_{11}$ otherwise.
If $|\mathcal{V}[v_{j},v_{l}]|=2$, then $G[v_1,v_n]$ properly contains $G_{10}$ if $v_jv_l\in E(G)$, and $d(v_j)=d(v_{j-1})=2$ if $v_jv_l\not\in E(G)$, in which case $G_1$ is properly contained in $G[v_1,v_n]$. In each case $\mathfrak{D}_2$ holds.

If $v_{i}v_{k}\in E(G)$ and $v_{k}v_{j}\in E(G)$, then we also look at $|\mathcal{V}[v_{j},v_{l}]|$.
If $|\mathcal{V}[v_{j},v_{l}]|=3$, then $\mathcal{V}[v_j,v_l]$ is  a path by the 2-connectedness of $G$, and thus $G[v_1,v_n]$ properly contains $G_{17}$ if $v_jv_l\in E(G)$, and $G_{5}$ otherwise.
If $|\mathcal{V}[v_{j},v_{l}]|=2$, then $G[v_1,v_n]$ properly contains $G_{15}$ if $v_jv_l\in E(G)$, and $d(v_j)=3,d(v_{j-1})=2$ if $v_jv_l\not\in E(G)$, in which case $G_2$ is properly contained in $G[v_1,v_n]$. In each case $\mathfrak{D}_2$ holds.

We assume now that $|\mathcal{V}[v_{k},v_{j}]|=2$.

If $v_{k}v_{j}\not\in E(G)$, then $d(v_{i+1})=2,d(v_{k})=3$ if $v_{i}v_{k}\in E(G)$, and $d(v_{i+1})=d(v_{k})=2$ if $v_{i}v_{k}\not\in E(G)$. Therefore, $G[v_{1},v_{n}]$ properly contain $G_{2}$ in the former case, and $G_{1}$ in the latter case.

If $v_{k}v_{j}\in E(G)$,  then we look at $|\mathcal{V}[v_{j},v_{l}]|$.

If $|\mathcal{V}[v_{j},v_{l}]|=3$, then $\mathcal{V}[v_j,v_l]$ is  a path by the 2-connectedness of $G$. Therefore, $G[v_1,v_n]$ properly contains $G_9$ if $|\{v_iv_k,v_jv_l\}\cap E(G)|=0$, $G_{14}$ if $|\{v_iv_k,v_jv_l\}\cap E(G)|=1$, and $G_{16}$ if $|\{v_iv_k,v_jv_l\}\cap E(G)|=2$. In each case $\mathfrak{D}_2$ holds.

If $|\mathcal{V}[v_{j},v_{l}]|=2$, then we consider two subcases. If $v_jv_l\not\in E(G)$, then $d(v_{i+1})=d(v_j)=2$ and $d(v_k)\leq 4$, which implies that $G[v_1,v_n]$ properly contains $G_3$. Adjusting the order of the vertices in the boundary of $G$ from $v_1,\ldots, v_{k-1},v_k,v_j,v_{j+1},\ldots,v_n$ to $v_1,\ldots v_{k-1},v_j,v_k,v_{j+1},\ldots,v_n$, we obtain an outer-1-planar drawing of $G+v_{i+1}v_j$, and thus $\mathfrak{D}_1$ holds.
If $v_jv_l\in E(G)$, then
$G[v_1,v_n]$ properly contains $G_{13}$ if $v_iv_k\in E(G)$, and $G_8$ otherwise. In each case $\mathfrak{D}_2$ holds.

\textbf{Subcase 4.4.} $|\mathcal{V}[v_{j},v_{l}]|=3$, and $|\mathcal{V}[v_{i},v_{k}]|,|\mathcal{V}[v_{k},v_{j}]|\leq 3$.

This is a symmetric case of Subcase 4.3, so we omit the proof here.

\textbf{Subcase 4.5.} $|\mathcal{V}[v_{i},v_{k}]|=|\mathcal{V}[v_{j},v_{l}]|=2$, and $|\mathcal{V}[v_{k},v_{j}]|\leq 3$.

If $v_iv_k,v_jv_l\in E(G)$, then $v_iv_j$ co-crosses $v_kv_l$, as desired. Hence in the following we assume  $|\{v_iv_k,v_jv_l\}\cap E(G)|\leq 1$. By symmetry, we assume that $v_iv_k\not\in E(G)$.

If $|\mathcal{V}[v_{k},v_{j}]|=3$, then by the  2-connectedness of $G$, $\mathcal{V}[v_k,v_j]$ is a path. If $v_kv_j\in E(G)$, then $d(v_{k+1})=2$, $d(v_{k})=3$ and thus $G_2$ is properly contained in $G[v_1,v_n]$. If $v_kv_j\not\in E(G)$, then $d(v_k)=d(v_{k+1})=2$ and thus $G_1$ is properly contained in $G[v_1,v_n]$.

If $|\mathcal{V}[v_{k},v_{j}]|=2$, then $v_kv_j\in E(G)$ by the  2-connectedness of $G$. If $v_jv_l\in E(G)$, then $d(v_k)=2$, $d(v_{j})=3$ and thus $G_2$ is properly contained in $G[v_1,v_n]$. If $v_jv_l\not\in E(G)$, then $d(v_k)=d(v_{j})=2$ and thus $G_1$ is properly contained in $G[v_1,v_n]$.
\end{proof}

We now come back to the proof for Case 4. By Claim A, we assume that $v_iv_j$ co-crosses $v_kv_l$ in $G$, as otherwise we have done the proof. Since $n\geq 7$, either $l\neq n$ or $i\neq 1$. We assume the former by symmetry.
If $d(v_{l})\leq7$ , then $G_{3}$ or $G_6$ or $G_7$ or $G_{12}$ is properly contained in $G[v_{1},v_{n}]$, and $\mathfrak{D}_1$, $\mathfrak{D}_2$ hold. Hence we assume $d(v_{l})\geq8$. Under this condition,
there is a chord $v_{l}v_{t}\in E(G)$ with $l<t\leq n$ or $1\leq t<i$.
If $1\leq t<i$, then $i\neq 1$ and thus $d(v_{i})\geq8$, as otherwise $G_{3}$ or $G_6$ or $G_7$ or $G_{12}$ is properly contained in $G[v_{1},v_{n}]$,  and $\mathfrak{D}_1$, $\mathfrak{D}_2$ hold. So, there is a chord $v_sv_i$ with $t\leq s<i$ (note that $v_lv_t$ can be crossed at most once).

Consequently, we have to consider the following subcases to complete the proof: (1) there is a chord $v_{l}v_{t}\in E(G)$ with $l<t\leq n$; (2) there is a chord $v_sv_i$ with $t\leq s<i$. We assume the former by symmetry, and meanwhile, assume that $t-l$ is as large as possible.

If $v_{l}v_{t}$ crosses $v_{a}v_{b}$ with $l<a<t$, then by Claim A, $v_lv_t$ co-crosses $v_av_b$, as otherwise we have finished the proof. This implies that $a=l+1$ and $b=t+1$. Since $d(v_{l})\geq8$, there is another chord $v_lv_s$ with $1\leq s<i$ or $b<s\leq n$. Without loss of generality, assume the latter. By Claim A, $v_lv_s$ is not crossed (note that $v_lv_s$ cannot be co-crossed by another edge). If $6\leq |\mathcal{V}[v_l,v_s]|<n$, then applying the induction hypotheses to the graph $G[v_l,v_s]$ (note that there is no edge between $\mathcal{V}(v_l,v_s)$ and $\mathcal{V}(v_s,v_l)$), we conclude that it properly contains one of the configurations among $G_{1}-G_{17}$ such that $\mathfrak{D}_1$ and $\mathfrak{D}_2$ hold, and so does $G[v_{1},v_{n}]$. If $|\mathcal{V}[v_l,v_s]|=5$, then by (2), $G[v_l,v_s]$ properly contains one of the configurations among $G_{1}-G_{4},G_6,G_8,G_{13}$  such that $\mathfrak{D}_1$ and $\mathfrak{D}_2$ hold, because the exclude structure mentioned in (2) cannot appear in $G[v_l,v_s]$.

On the other hand, suppose that $v_{l}v_{t}$ is not crossed. If $|\mathcal{V}[v_l,v_t]|\geq 6$, then applying the induction hypotheses to the graph $G[v_l,v_t]$, we conclude that it properly contains one of the configurations among $G_{1}-G_{17}$  such that $\mathfrak{D}_1$ and $\mathfrak{D}_2$ hold, and so does $G[v_{1},v_{n}]$. Hence we assume $|\mathcal{V}[v_l,v_t]|\leq 5$,

If there is a chord $v_lv_s$ with $1\leq s<i$, then $v_lv_s$ is not crossed by Claim A (note that $v_lv_s$ cannot be co-crossed by another edge). If $|\mathcal{V}[v_s,v_l]|\geq 6$, then applying the induction hypotheses to the graph $G[v_s,v_l]$, we conclude that it properly contains one of the configurations among $G_{1}-G_{17}$  such that $\mathfrak{D}_1$ and $\mathfrak{D}_2$ hold, and so does $G[v_{1},v_{n}]$. If $|\mathcal{V}[v_s,v_l]|=5$, then by (2), $G[v_s,v_l]$ properly contains one of the configurations among $G_{1}-G_{4},G_6,G_8,G_{13}$  such that $\mathfrak{D}_1$ and $\mathfrak{D}_2$ hold, because the exclude structure mentioned in (2) cannot appear in $G[v_s,v_l]$. Therefore, we assume that there is no such a chord $v_lv_s$ with $1\leq s<i$.

Since $|\mathcal{V}[v_l,v_t]|\leq 5$ and $t$ is an integer such that  $t-l$ is as large as possible, $d(v_l)\leq 7$, contradicting the assumption $d(v_l)\geq 8$.

This ends the proof of (3).
\end{proof}

An \emph{end-block} of a connected graph $G$ of minimum degree at least 2 is a 2-connected subgraph containing exactly one cut-vertex of $G$ if $G$ has cut-vertices (i.e, $G$ is not 2-connected), or $G$ itself if $G$ is 2-connected.

\begin{thm}\label{str-1}
Each outer-$1$-plane graph $G$ with $\delta(G)\geq 2$ contains at least one of the configurations among $G_{1}-G_{17}$ such that $\mathfrak{D}_1$ and $\mathfrak{D}_2$ hold.
\end{thm}

\begin{proof}
Theorem \ref{2-connected}(3) implies this result for the case that
$G$ is 2-connected and $|G|\geq 6$. If $G$ is not 2-connected or $|G|\leq 5$, then let $H$ be an end-block of $G$.
Let $v_1,v_2,\ldots,v_n$ be the vertices of $H$ with a clockwise sequence in the drawing, where $n=|H|$ and only $v_1$ may be a cut-vertex. Since $\delta(G)\geq 2$, $n\geq 3$.

If $n\geq 6$, then by Theorem \ref{2-connected}(3), one of the configurations among $G_{1}-G_{17}$ is properly contained in
$K[v_1,v_n]$ such that $\mathfrak{D}_1$ and $\mathfrak{D}_2$ hold. Hence $G$ contains one of the configurations among $G_{1}-G_{17}$  such that $\mathfrak{D}_1$ and $\mathfrak{D}_2$ hold.

If $n=5$, then by Theorem \ref{2-connected}(2), $H[v_1,v_n]$ properly contains (and thus $G$ contains) one configuration among $G_{1}-G_{4},G_6,G_8,G_{13}$  such that $\mathfrak{D}_1$ and $\mathfrak{D}_2$ hold unless $\mathcal{V}[v_1,v_5]$ is a path and $v_1v_4,v_2v_5\in E(G)$, in which case $d(v_5)\leq 3$ and thus $G_7$ or $G_{12}$ is contained in $G$  such that $\mathfrak{D}_2$ holds.

If $n=4$, then by Theorem \ref{2-connected}(1), $H[v_1,v_n]$ properly contains (and thus $G$ contains) $G_1$ or $G_2$ unless $v_1v_3$ crosses $v_2v_4$ and $v_1v_2,v_3v_4\in E(G)$, in which case $d(v_4)\leq 3$ and thus $G_3$ or $G_{6}$ is contained in $G$. Clearly, $\mathfrak{D}_1$ and $\mathfrak{D}_2$ hold.

If $n=3$, then by the 2-connectedness of $H$, $v_1v_2,v_2v_3,v_3v_1\in E(H)$ and $d(v_2)=d(v_3)=2$. This implies that $G$ contains $G_1$.
\end{proof}

In the following, we first use Theorem \ref{str-1} to deduce the theorem that was recently proved in \cite{ZLLZ}.

\begin{cor}\label{cor}
$(1)$ Each outer-$1$-planar graph $G$ with $\delta(G)\geq 2$ has an edge $uv$ such that $d(u)+d(v)\leq 9$; \\
$(2)$ Each maximal outer-$1$-planar graph $G$ has an edge $uv$ such that $d(u)+d(v)\leq 7$.
\end{cor}

\begin{proof}
(1) Since each configuration $G_i$ with $i\neq 6$ in Fig.\,1 contains an edge $uv$ with $d(u)=2$ and $d(v)\leq 7$, and the configuration $G_6$ contains an edge $uv$ with $d(u)=d(v)=3$, by Theorem \ref{str-1}, $G$ contains an edge $uv$ such that $d(u)+d(v)\leq 9$.

(2) If $G$ is a maximal outer-1-planar graph, then it is easy to see that $\delta(G)\geq 2$. If $G$ contains $G_3$ in Fig.\,\ref{structure-1}, then by $\mathfrak{D}_1$, $G+uv$ is outer-1-planar, contradicting the maximality of $G$. Hence by Theorem \ref{str-1}, $G$ contains one of the configurations among $G_{1},G_2,G_4-G_{17}$, in each of which configuration there is an edge $uv$ with either $d(u)=2$ and $d(v)\leq 5$, or $d(u)=d(v)=3$. Hence $G$ contains an edge $uv$ such that $d(u)+d(v)\leq 7$.
\end{proof}

The following is another immediate corollary from Theorem \ref{str-1}, which will be used in Section \ref{sec-3} to prove an interesting result on the list 3-dynamic coloring of outer-1-planar graphs.

\begin{cor}\label{use-to-color}
Each outer-$1$-planar graph contains one of the following configurations:
\begin{description}
  \item[(1)] a vertex of degree at most $1$;
  \item[(2)] two adjacent vertices of degree $2$;
  \item[(3)] a triangle incident with a vertex of degree $2$;
  \item[(4)] the configuration $G_i$ as in Fig.\,$\ref{structure-1}$, where $i\in \{3,6,7,8,9,10,11\}$. \hfill$\square$
\end{description}
\end{cor}

To end this section, we show that the list of the configurations in Theorem \ref{str-1} is minimal in the sense that for each configuration $G_i$ with $1\leq i\leq 17$,
there are outer-1-planar graphs containing $G_i$ that do not have any of another sixteen configurations.

\begin{figure}
  \centering
  \includegraphics[width=10cm]{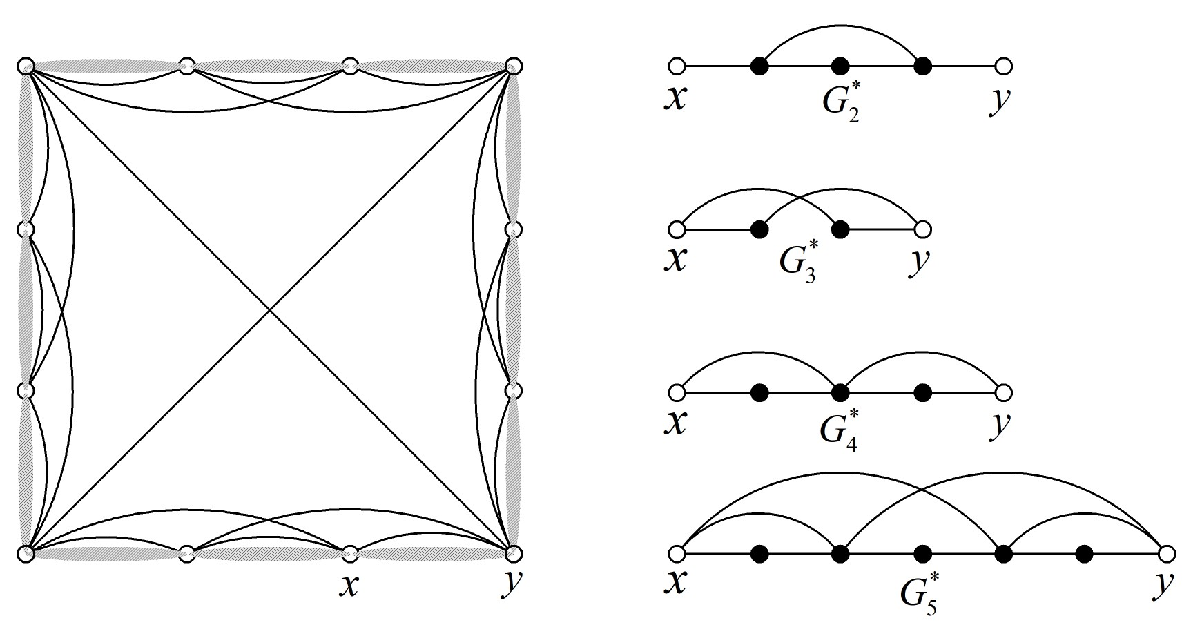}
  \caption{The construction of special outer-1-planar graphs}\label{minimal}
\end{figure}

Trivially, a cycle contains $G_1$ and does not contain $G_i$ for any $2\leq i\leq 17$. We now look at the left picture in Fig.\,\ref{minimal}. Into each of the shadowed areas, we embed the configurations $G^*_i$ with $2\leq i\leq 17$ so that $x$ and $y$ are end
vertices (we do not care about the direction of the embedding of the configuration, although some configuration, say $G_{14}$ for example, is not symmetric in its drawing), and denote the resulting graph by $H_i$. Here, $G^*_i$ with $2\leq i\leq 5$ is shown as in Fig.\,\ref{minimal}, and $G^*_i$ with $6\leq i\leq 17$ corresponds to $G_i$ in Fig.\,\ref{structure-1}. It is easy to check that the graph $H_i$ with $2\leq i\leq 17$ is an outer-1-planar graph that contains $G_i$ and does not contain $G_j$ with $j\neq i$. This also implies the sharpness of Corollary \ref{cor}.

\section{List 3-Dynamic Coloring}\label{sec-3}

A \emph{proper $k$-coloring} $c$ of a graph $G$ is a function from its vertex set $V(G)$ to $\{1,2,\ldots,k\}$ such that $c(u)\neq c(v)$ if $u$ is adjacent to $v$.
An \emph{$r$-dynamic $k$-coloring} of a graph $G$ is a proper $k$-coloring such that for any vertex $v$, there are at least $\min\{r,d(v)\}$ distinct colors appearing on the neighbors of $v$. The minimum integer $k$ so that $G$ has a proper $k$-coloring or an $r$-dynamic $k$-coloring is the \emph{chromatic number} or \emph{$r$-dynamic chromatic number} of $G$, denoted by $\chi(G)$ or $\chi_r(G)$, respectively. Clearly, $\chi_r(G)\geq \chi_1(G)=\chi(G)$, where $r\geq 1$.

The notion of $r$-dynamic coloring was introduced by \cite{M01}, newly
studied by \cite{JKO16,KMW15,KUAD20,ZB18}, and also investigated under
the notion of $r$-hued coloring, see
\cite{CFL20,CLL18,L21,MHK18,SFC14,SL18,SLW16,ZCM17}. As starting cases
of $r$-dynamic coloring, the $2$-dynamic coloring (known as dynamic
coloring in literature), see
\cite{A11,AGJ,AGJ14-1,AGJ14-2,BS12,CFL12,KLO,LMR18,MLG16,MMS06}, and
the $3$-dynamic coloring, see \cite{AKK18,KP18,LL17}, have been considered.
The list analogue of dynamic coloring was introduced by \cite{AGJ}, and investigated by many authors
including \cite{E10,Gu21,KLP,KP18,KP,LL17,ZCM17,ZL21,ZB20}.

Suppose that a set $L(v)$ of colors, called a \emph{list} of $v$, is assigned to each vertex $v\in V(G)$.
An \emph{$r$-dynamic $L$-coloring} of $G$ is an $r$-dynamic coloring $c$ so that $c(v)\in L(v)$ for every $v\in V(G)$. A graph $G$ is \emph{$r$-dynamic $k$-choosable} if
$G$ has an $r$-dynamic $L$-coloring whenever $|L(v)|=k$ for every $v\in V(G)$. The minimum integer $k$ for which $G$ is $r$-dynamic $k$-choosable is the \emph{list $r$-dynamic chromatic number} of $G$, denoted by $ch_r(G)$. It is obvious that $ch_r(G)\geq \chi_r(G)$.

\begin{figure}
  \centering
  \includegraphics[width=4cm]{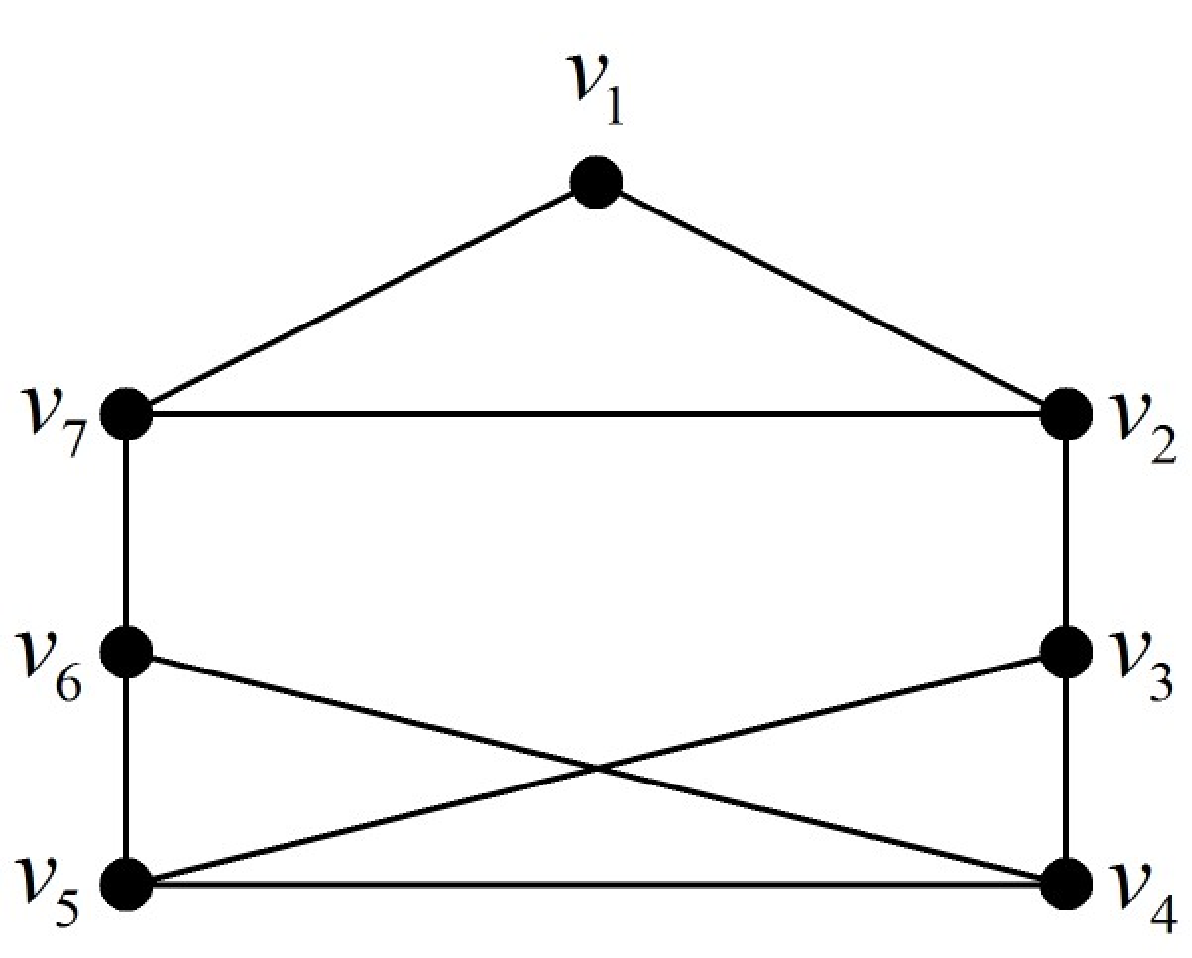}
  \caption{An outer-1-planar graph with $3$-dynamic chromatic number 6}\label{o1p-6}
\end{figure}

In this section, we apply the structural theorem obtained in Section \ref{sec-2} (precisely, Corollary \ref{use-to-color}) to prove that the list $3$-dynamic chromatic number of every outer-1-planar graph is at most 6, and moreover, this upper bound 6 is sharp because of the existence of an outer-1-planar graph with $3$-dynamic chromatic number 6, see Theorem \ref{sharp}.

\begin{thm}\label{sharp}
  There exists an outer-1-planar graph with $3$-dynamic chromatic number $6$.
\end{thm}

\begin{proof}
  Look at the outer-1-planar graph $G$ in Fig.\,\ref{o1p-6}. We claim that its $3$-dynamic chromatic number is exactly 6. Since $v_3$ has degree 3 and $v_2,v_4,v_5$ are its neighbors, those four vertices have distinct colors in any 3-dynamic coloring $G$. Without loss of generalization, assume that $v_2,v_3,v_4$ and $v_5$ are colored with 1,2,3 and 4, respectively.
It is clear that $v_6$ cannot be colored by 2 or 3 (otherwise two neighbors of $v_5$, which has degree 3, are monochromatic), and also cannot be colored by 1 (otherwise two neighbors of $v_7$, which has degree 3, are monochromatic). Therefore, we assume that $v_6$ is colored with 5 (note that the color 4 is forbidden on $v_6$ since it is adjacent to $v_5$ that has color 4). At this stage, the colors $1,2,3,4$ and $5$ are forbidden on $v_7$ (otherwise two adjacent vertices receive a same color, or a vertex of degree 3 has two monochromatic neighbors). Hence we have to color $v_7$ with 6, and then color $v_1$ with 3. This implies that  the $3$-dynamic chromatic number of $G$ is exactly 6.
\end{proof}

\begin{thm}\label{list-3-dynamic}
If $G$ is an outer-$1$-planar graph, then $ch_3(G)\leq 6$.
\end{thm}
\begin{proof}
Let $G$ be a counterexample to the theorem statement with the smallest number of vertices. It follows that there exists a list assignment $L$ of size 6 such that $G$ has no 3-dynamic $L$-coloring and any proper subgraph of $G$ is 3-dynamic $L$-colorable. Clearly, $G$ is connected. In what follows, we prove a series of propositions contradicting Corollary \ref{use-to-color} to complete the proof.

\begin{description}
  \item[(1)] $\delta(G)\geq 2$.
\end{description}

Suppose, to the contrary, that there is an edge $uv$ with $d(u)=1$. By the minimality of $G$, the graph $G'=G-u$ has a 3-dynamic $L$-coloring $c$. It is easy to see that $d(v)\geq 2$, because otherwise $G$ is exactly $K_2$ that is 3-dynamic $L$-colorable, a contradiction. If $d(v)\geq 4$, then $v$ has degree at least 3 in $G'$ and thus $v$ is incident with at least three distinct colors in $c$. In this case we color $u$ from its list with a color different from $c(v)$, and then obtain a 3-dynamic $L$-coloring of $G$, a contradiction. If $d(v)\leq 3$, then color $u$ from its list with a color that is different from the colors used on $v$ and its neighbor(s) in $G'$. This also constructs a 3-dynamic $L$-coloring of $G$, a contradiction.

\begin{description}
  \item[(2)] \textit{$G$ does not contain two adjacent vertices of degree $2$.}
\end{description}

Suppose, to the contrary, that there is an edge $uv$ with $d(u)=d(v)=2$. By the minimality of $G$, the graph $G'=G-\{u,v\}$ has a 3-dynamic $L$-coloring $c$. By $x$ and $y$, we denote the other neighbor of $u$ and $v$ besides $v$ and $u$, respectively.

Assume first that $d(x)\geq 4$, then color $u$ with $c(u)\in L(u)\backslash \{c(x),c(y)\}$. If $d(y)\geq 4$, then color $v$ with $c(v)\in L(v)\backslash \{c(x),c(u),c(y)\}$. If $d(y)\leq 3$, then color $v$ from its list with a color that is different from $c(x)$, $c(u)$, $c(y)$ and the colors (at most two) used on the neighbor(s) of $y$ in $G'$. In each case, at most five colors are forbidden and we have six available colors for $v$. Hence we obtain a 3-dynamic $L$-coloring of $G$, a contradiction.

Second, assume that $d(x)\leq 3$, and by symmetry, that $d(y)\leq 3$. Coloring $u$ with a color $c(u)$ from its list that is different from $c(x),c(y)$ and the colors (at most two) used on the neighbor(s) of $x$ in $G'$, and then coloring $v$ from its list with a color that is different from $c(x),c(y),c(u)$ and the colors (at most two) used on the neighbor(s) of $y$ in $G'$, we construct a 3-dynamic $L$-coloring of $G$, a contradiction.

\begin{description}
  \item[(3)] \textit{$G$ does not contain a triangle $xuyx$ in $G$ with $d(u)=2$.}
\end{description}

Suppose, to the contrary, that $G$ contains a triangle $xuyx$ with $d(u)=2$.
By the minimality of $G$, $G'=G-\{u\}$ has a 3-dynamic $L$-coloring $c$. By (2), $d(x)\geq 3$ and $d(y)\geq 3$.
Assume first that $d(x)\geq 4$. If $d(y)\geq 4$, then color $u$ from its list with a color different from $c(x)$ and $c(y)$. If $d(y)=3$, then color $u$ from its list with a color different from $c(x),c(y)$ and $c(y_1)$, where $y_1$ is the third neighbor of $y$ other than $x$ and $u$.
In each case,
we obtain a 3-dynamic $L$-coloring of $G$, a contradiction.
Second, assume that $d(x)=3$, and by symmetry, that $d(y)=3$.
Let $x_1$ be the neighbor of $x$ other than $u$ and $y$, and let $y_1$ be the neighbor of $y$ other than $u$ and $x$. We
color $u$ with $c(u)\in L(u)\backslash \{c(x),c(y),c(x_1),c(y_1)\}$, and then obtain a 3-dynamic $L$-coloring of $G$, a contradiction. Note that $c(x)\neq c(y)$ since $xy\in E(G')$.

\begin{description}
  \item[(4)] \textit{$G$ does not contain the configuration $G_3$.}
\end{description}

Suppose, to the contrary, that $G$ contains a copy of $G_{3}$ as in Fig.\,\ref{structure-1}.
By the minimality of $G$, $G'=G-\{u\}$ has a 3-dynamic $L$-coloring $c$. By (2), we have $d(x)\geq 3$ and $d(y)\geq 3$. Assume first that $d(x)\geq 4$. If $d(y)\geq 4$, then color $u$ from its list with a color different from $c(x)$ and $c(y)$. If $d(y)=3$, then color $u$ from its list with a color different from $c(x),c(y),c(v)$ and $c(y_1)$, where $y_1$ is the third neighbor of $y$ other than $v$ and $u$. In each case, we obtain a 3-dynamic $L$-coloring of $G$, a contradiction. Second, assume that $d(x)=3$, and by symmetry, that $d(y)=3$. Let $x_1$ be the neighbor of $x$ other than $u$ and $v$, and let $y_1$ be the neighbor of $y$ other than $u$ and $v$. We color $u$ with $c(u)\in L(u)\backslash \{c(x),c(y),c(x_1),c(y_1),c(v)\}$, and then obtain a 3-dynamic $L$-coloring of $G$, a contradiction. Note that $c(x)\neq c(y)$ since $x$ and $y$ are the only two neighbors of $v$ in $G'$.

\begin{description}
  \item[(5)] \textit{$G$ does not contain the configuration $G_6$.}
\end{description}

Suppose, to the contrary, that $G$ contains a copy of $G_{6}$ as in Fig.\,\ref{structure-1}.
By the minimality of $G$, $G'=G-\{u\}$ has a 3-dynamic $L$-coloring $c$. By (3), we have $d(x)\geq 3$ and $d(y)\geq 3$. Assume first that $d(x)\geq 4$. If $d(y)\geq 4$, then color $u$ from its list with a color different from $c(x)$, $c(y)$ and $c(v)$. If $d(y)=3$, then color $u$ from its list with a color different from $c(x),c(y),c(v)$ and $c(y_1)$, where $y_1$ is the third neighbor of $y$ other than $v$ and $u$. In each case, we obtain a 3-dynamic $L$-coloring of $G$, a contradiction.
Second, assume that $d(x)=3$, and by symmetry, that $d(y)=3$. Let $x_1$ be the neighbor of $x$ other than $u$ and $v$, and let $y_1$ be the neighbor of $y$ other than $u$ and $v$. We color $u$ with $c(u)\in L(u)\backslash \{c(x),c(y),c(x_1),c(y_1),c(v)\}$, and then obtain a 3-dynamic $L$-coloring of $G$, a contradiction. Note that $c(x),c(y)$ and $c(v)$ are pairwise different since $v$ has only two neighbors $x$ and $y$ in $G'$.

\begin{description}
  \item[(6)] \textit{$G$ does not have a chordless quadrilateral $vuwxv$ such that $d(v) = d(w) = 3$, $d(u) = 2$, $v, u,w$ appear in the outer boundary of $G$ consecutively in that order, and $G-\{u\}+vw$ is outer-$1$-planar. This directly implies that $G$ does not contain the configuration $G_7$ or $G_{10}$.}
\end{description}

If such a quadrilateral exists, then $G-\{u\}+vw$ has a 3-dynamic $L$-coloring $c$ by the minimality of $G$.
By choosing a color for $u$ from
$L(u) \setminus \{c(x), c(v), c(w), c(v_1), c(w_1)\}$, we obtain a 3-dynamic $L$-coloring of
$G$, where $v_1$ is the neighbor of $v$ other than $u, x$ and $w_1$ is the neighbor
of $w$ other than $u, x$ (possibly $v_1 = w_1$).


\begin{description}
  \item[(7)] \textit{$G$ does not contain the configuration $G_8$.}
\end{description}

Suppose, to the contrary, that $G$ contains a copy of $G_{8}$ as in Fig.\,\ref{structure-1}.
By the minimality of $G$, $G'=G-\{u,w\}$ has a 3-dynamic $L$-coloring $c$.  By (2) and (3),  $d(x)\geq 3$ and $d(y)\geq 3$. By (6), $d(x)\geq 4$. If $d(y)\geq 4$, then color $w$ with $c(w)\in L(w)\backslash \{c(x),c(y),c(v)\}$ and $u$ with
$c(u)\in L(u)\backslash \{c(x),c(y),c(w),c(v)\}$. If $d(y)=3$, then color $w$ with $c(w)\in L(w)\backslash \{c(x),c(y),c(v),c(y_1)\}$ and $u$ with $c(u)\in L(u)\backslash \{c(x),c(y),$\\$c(w),c(v)\}$, where $y_1$ is the third neighbor of $y$ other than $v$ and $w$.
In each case, we obtain a 3-dynamic $L$-coloring of $G$, a contradiction.

\begin{description}
  \item[(8)] \textit{$G$ does not contain the configuration $G_9$.}
\end{description}

Suppose, to the contrary, that $G$ contains a copy of $G_{9}$ as in Fig.\,\ref{structure-1}.
By the minimality of $G$, $G'=G-\{u,w\}$ has a 3-dynamic $L$-coloring $c$.  By (2),  $d(x)\geq 3$ and $d(y)\geq 3$. By (6), $d(x)\geq 4$. If $d(y)\geq 4$, then color $w$ with $c(w)\in L(w)\backslash \{c(x),c(y),c(v),c(z)\}$ and $u$ with
$c(u)\in L(u)\backslash \{c(x),c(y),c(w),c(v)\}$. If $d(y)=3$, then color $w$ with $c(w)\in L(w)\backslash \{c(x),c(y),c(v),c(y_1),c(z)\}$ and $u$ with $c(u)\in L(u)\backslash \{c(x),$\\$c(y),c(w),c(v)\}$, where $y_1$ is the third neighbor of $y$ other than $z$ and $w$.
In each case, we obtain a 3-dynamic $L$-coloring of $G$, a contradiction.

\begin{description}
  \item[(9)] \textit{$G$ does not contain the configuration $G_{11}$.}
\end{description}

Suppose, to the contrary, that $G$ contains a copy of $G_{11}$ as in Fig.\,\ref{structure-1}. By (2), $d(x)\geq 3$ and $d(y)\geq 3$.
By the minimality of $G$,
$G'=G-\{u,v,a\}$ has a 3-dynamic $L$-coloring $c$.
If $d(x)\geq 4$ and $d(y)\geq 4$, then color $u,v$ and $a$ in this order with
$c(u)\in L(u)\backslash  \{c(x),c(y),c(z),c(w)\}$, $c(v)\in L(v) \backslash  \{c(x),c(y),c(u),$\\$c(w)\}$ and $c(a)\in L(a)\backslash \{c(x),c(y),c(u),c(v)\}$, respectively.
If $d(x)\geq 4$ and $d(y)=3$ (the case when $d(x)=3$ and $d(y)\geq 4$ is similar),
then color $u,a$ and $v$ in this order with
$c(u)\in L(u)\backslash  \{c(x),c(y),c(z),$\\$c(w)\}$, $c(a)\in L(a)\backslash \{c(w),c(y),c(y_1),c(u),c(x)\}$ and
$c(v)\in L(v) \backslash  \{c(x),c(y),c(u),c(w),c(a)\}$, respectively, where $y_1$ is the third neighbor of $y$ other than $w$ and $a$. In each case, we obtain a 3-dynamic $L$-coloring of $G$, a contradiction. Hence in the following, we assume that $d(x)=d(y)=3$.

Let $x_1$ be the third neighbor of $x$ other than $v$ and $z$, and let $y_1$ be the third neighbor of $y$ other than $w$ and $a$.
By the minimality of $G$, $G'=G-\{u,v,a\}$ has a 3-dynamic $L$-coloring $c$. Color $v$ with $c(v)\in L(v)\backslash \{c(x),c(x_1),c(z),c(w),c(y)\}$ and $a$ with $c(a)\in L(a)\backslash \{c(v),c(y),c(y_1),c(w),c(x)\}$.
If there is a color available for $u$ that is different from the colors used on $x,z,w,v,a$ and $y$, then use it to color $u$ and we immediately obtain a 3-dynamic $L$-coloring of $G$, a contradiction. So, the worst case here is the colors on $x,z,w,v,a$ and $y$ are rainbow, say $1,2,3,4,5$ and 6, and moreover, $L(u)=\{1,2,3,4,5,6\}$.

At this stage, we color $u$ with 2 (resp.\,5) and then try to recolor $z$ (resp.\,$a$). If it is possible to recolor $z$ (resp.\,$a$) with a color different from $1,2,3,4,6$ (resp.\,$1,3,4,5,6$), and different from the color on $x_1$ (resp.\,$y_1$), then we obtain a 3-dynamic $L$-coloring of $G$, a contradiction. So, the difficult case is that $L(z)=\{1,2,3,4,6,c(x_1)\}$ and $L(a)=\{1,3,4,5,6,c(y_1)\}$, where $3\not\in \{c(x_1),c(y_1)\}$. We now recolor $z$ and $a$ with 3, and
recolor $w,v$ and $u$ in this order with $c(w)\in L(w)\backslash \{1,3,6,c(y_1)\}$, $c(v)\in L(v)\backslash \{1,3,6,c(w),c(x_1)\}$ and $c(u)\in L(u)\backslash \{1,3,6,c(v),c(w)\}$, respectively. It is easy to see that the resulting coloring of $G$ is a 3-dynamic $L$-coloring, a contradiction.
\end{proof}



\newcommand{\etalchar}[1]{$^{#1}$}

\end{document}